\documentclass [11pt]{amsart}
\usepackage{amscd, amsmath,amsthm,amssymb, amsfonts, epsfig}
\usepackage{bbm}
\usepackage{graphicx}
\usepackage[utf8]{inputenc}
\usepackage{xcolor}
\usepackage{mathrsfs}  
\usepackage{esvect}
\usepackage{enumitem}
\usepackage{indentfirst}
\usepackage{caption}
\usepackage{subfigure,tikz}
\usepackage{tkz-euclide}
\usepackage{comment}
\usepackage{outlines}
\usepackage{multicol}
\usepackage{pgfplots}
\usepgfplotslibrary{fillbetween}
\usepackage{comment}

\usetikzlibrary{decorations.markings}

\usetikzlibrary{intersections,positioning, calc, patterns}

\numberwithin{equation}{section}

\newtheorem{theorem}{Theorem}[section]
\newtheorem{corollary}[theorem]{Corollary}
\newtheorem{lemma}[theorem]{Lemma}

\newtheorem{conjecture}{Conjecture}

\pgfplotsset{compat=1.18}
\begin{document}

\title[First eigenvalue of free boundary minimal surfaces of genus zero]{A symmetry condition for genus zero free boundary minimal surfaces attaining the first eigenvalue of one}
\author{Dong-Hwi Seo}
\address{Research Institute of Mathematics, Seoul National University, 1 Gwanak-ro, Gwanak-gu, Seoul 08826, Republic of Korea}
\email{donghwi.seo26@gmail.com}


\begin{abstract}
    An embedded free boundary minimal surface in the 3-ball has a Steklov eigenvalue of one due to its coordinate functions. Fraser and Li conjectured that whether one is the first nonzero Steklov eigenvalue. In this paper, we show that if an embedded free boundary minimal surface of genus zero, with $n$ boundary components, in the 3-ball has $n$ distinct reflection planes, then one is the first eigenvalue of the surface. 
\end{abstract}

\maketitle

\section{Introduction}
Let $\mathbb{B}^3$ be the standard Euclidean 3-ball. A surface $M^2$ with boundary $\partial M$ is a free boundary minimal surface in $\mathbb{B}^3$ if the mean curvature of $M$ is everywhere zero and $M$ meets $\partial \mathbb{B}^3$ perpendicularly along $\partial M$. This is equivalent to the requirements that $M$ be a critical point of the area functional among all variations of $M$ in $\mathbb{B}^3$ whose boundaries are contained in $\partial \mathbb{B}^3$. For general references to this topic, we refer the reader to \cite{Li20FBM, Car2019FBM}. 

Free boundary minimal surfaces in $\mathbb{B}^3$ have an interesting connection with the Steklov eigenvalue problem, initially investigated by Fraser and Scheon \cite{FS11FSE}. The Steklov eigenvalue problem \cite{Ste1902problemes} of a smooth bounded domain $\Omega$ in a Riemannian manifold is to find $\sigma\in \mathbb{R}$ for which there exists a function $u\in C^{\infty}(\Omega)$ satisfying
\begin{align} \label{problem}
\left\{
\begin{array}{rcll}
     \Delta u &=& 0&   \text{in   } \Omega  \\
     \frac{\partial u}{\partial \eta} &=& \sigma u &\text{on   } \partial \Omega
\end{array} \right.
,
\end{align}
where $\eta$ is the outward unit conormal vector of $\Omega$ along $\partial \Omega$. We call $u$ an Steklov eigenfunction and $\sigma$ an Steklov eigenvalue. It is known that the eigenvalues form a sequence $0=\sigma_0<\sigma_1\le\dots \rightarrow \infty$. See \cite{GP17SGS, CGGS2024SRD} for general surveys. Fraser and Schoen observed that every coordinate function of a free boundary minimal surface in $\mathbb{B}^n$ satisfies (\ref{problem}) with $\sigma=1$ \cite{FS11FSE}. Thus, the coordinate functions are Steklov eigenfunctions with eigenvalue 1. 

In \cite{FL14CSE}, Fraser and Li proposed the following conjecture.

\begin{conjecture}\label{conj: Fraser_Li} Let $N$ be an embedded free boundary minimal hypersurface in $\mathbb{B}^n$. Then, the first Steklov eigenvalue $\sigma_1(N)$ is 1. 
\end{conjecture}
Note that without the embeddedness assumption, there are some examples with $\sigma_1(M)<1$ by Fern\'andez-Hauswirth-Mira \cite{FHM2023FBM} and Kapouleas-McGrath \cite{KM2022FBM}. This conjecture can be seen as an analog of Yau's conjecture for closed minimal hypersurfaces in $\mathbb{S}^n$ \cite{yau1982PSS}. See surveys in \cite{Cho2006MSS, Bre2013MSS} and related results in \cite{CS2009FES, TY2013IFY1, TXY2014IFY2}. 

Conjecture \ref{conj: Fraser_Li} has been studied by several authors. McGrath showed that if an embedded free boundary minimal annulus in $\mathbb{B}^3$ has the reflection symmetries by three orthogonal planes passing through the origin, then the surface supports the conjecture \cite{McG18ACC}. In this work and his following work with Kusner \cite{KM20arXiv}, he showed that the conjecture holds for various classes of embedded free boundary minimal surfaces in $\mathbb{B}^3$. In particular, Kusner and McGrath showed that if an embedded free boundary minimal annulus has antipodal symmetry, the surface satisfies the conjecture.  Additionally, the author showed that the conjecture holds for several symmetric free boundary minimal annuli \cite{Seo2021SSC}. For example, if an embedded free boundary minimal annulus has two distinct reflection planes, then the surface has the first eigenvalue of one \cite[Theorem 1.2]{Seo2021SSC}. Here, we remark that there were no restrictions on reflection planes.

The main theorem in this paper is as follows. 
\begin{theorem} \label{thm: n symmetries}
Let $\Sigma$ be an embedded free boundary minimal surface of genus zero with $n$ boundary components in the unit ball $\mathbb{B}^3$. If $\Sigma$ has $n$ reflection symmetries, then $\sigma_1(\Sigma)=1$.
\end{theorem}

The main idea of the proof involves associating the boundary components with $n$ distinct points in $\mathbb{S}^2$ (Corollary \ref{cor: flux components}). These points are derived from (\ref{eqn: flux}), which is motivated by the flux of a minimal surface. This idea, combined with a basic linear algebraic observation, simplifies the analysis of reflections (for example, see Corollary \ref{cor: f_i and a boundary component}). Building on established techniques for genus zero free boundary minimal surfaces (Section \ref{sec:preliminaries}), we can obtain the main theorem.

This theorem can be viewed as an extension of the author's previous work on free boundary minimal annuli \cite{Seo2021SSC} to free boundary minimal surfaces of genus zero with $n$ boundary components. Since there are no restrictions on reflection planes, the theorem includes a new symmetry criterion on a surface to achieve the first Steklov eigenvalue of one. For example, the theorem explains that if $\Sigma$ has $n$ distinct reflection planes sharing a common axis, then $\sigma_1(\Sigma)=1$.

\subsection*{Acknowledgments} This work was supported by the National Research Foundation of Korea (NRF) grant funded by the Korean Government (MSIT) (No. 2021R1C1C2005144).

\section{Preliminaries}\label{sec:preliminaries}

\subsection{General properties}
Let $M$ be a compact embedded free boundary minimal surface of any topological type in $\mathbb{B}^3$.  In this paper, we define a coordinate plane as a plane passing through the origin.

\begin{theorem}[two-piece property, {\cite[Theorem A]{LM21TPP}}] \label{thm: two-piece property} Suppose $M$ is not homeomorphic to a disk. Then, a coordinate plane dissects $M$ by exactly two components.
\end{theorem}

From now on, a reflection means a mapping from $\mathbb{R}^3$ to itself that is an isometry with a plane as a set of fixed points. We denote by $R_{\Pi}$ the reflection through a plane $\Pi$.
\begin{lemma} \label{lemm: prop of a reflection}
\begin{enumerate}
    \item \label{lemm: prop of a reflection: identical}  Let $p_1, p_2 \in \mathbb{S}^2$ be two distinct points. If two reflections map $p_1$ into $p_2$, then the two reflection planes are identical.
    \item \label{lemm: prop of a reflection: origin} A reflection plane of $M$ passes through the origin.
\end{enumerate}
\end{lemma}
\begin{proof}
For (\ref{lemm: prop of a reflection: identical}), let $\Pi_1$ and $\Pi_2$ be such two reflection planes. Clearly, $p_1 \notin \Pi_1, \Pi_2$ and $R_{\Pi_1}^{-1}\circ R_{\Pi_2}$ is a rotation about an axis. The axis is the intersection of $\Pi_1$ and $\Pi_2$, so it does not contain $p_1$. Since, $R_{\Pi_1}\circ R_{\Pi_2}^{-1}(p_1)=p_1$, $R_{\Pi_1}^{-1}\circ R_{\Pi_2}$ is the identical map. Thus $\Pi_1=\Pi_2$. For (\ref{lemm: prop of a reflection: origin}), let $\Pi$ be a reflection plane of $M$. If $M$ is an equatorial disk, $\Pi$ passes through the origin. Otherwise, we can take two distinct points $p_1, p_2\in \mathbb{S}^2$ such that $R_{\Pi}(p_1)=p_2$ and $p_2 \neq -p_1$. Using (\ref{lemm: prop of a reflection: identical}), $\Pi$ should be a coordinate plane passing through the midpoint of $p_1$ and $p_2$, which completes the proof.
\end{proof}

\subsection{Properties for genus zero}
In the proof of the following lemmas, we use the fact that $\Sigma$ has genus zero.

\begin{lemma} $\Sigma$ satisfies the followings.
\label{lemm: properties of interior}
\begin{enumerate}
    \item (transversality) If a coordinate plane meets $\Sigma$, it intersects transversely. In particular, $o\notin \Sigma$, where $o$ is the origin. \label{lemm: properties of interior: transversality}
    \item (radial graphicality) If a ray from the origin meets $\Sigma$, it intersects at a single point transversely. \label{lemm: properties of interior: radial graphicality}
\end{enumerate}
\end{lemma}
\begin{proof}
    Every statement can be found in the proof of \cite[Proposition 8.1]{FS16SEB}. For detailed proof, see \cite[Corollary 4.5]{Seo2021SSC} for (\ref{lemm: properties of interior: transversality}) and \cite[Lemma 2.1]{MZ2024OAG} for (\ref{lemm: properties of interior: radial graphicality}). 
\end{proof}

We denote by
\begin{align*}
    (\partial \Sigma)_1, \dots, (\partial \Sigma)_n,
\end{align*}
the components of $\partial \Sigma$.

Let $x$ be the position vector in $\mathbb{R}^3$. 
For $i=1,\dots, n$, let 
\begin{align}\label{eqn: F_i, f_i}
    F_i:= \int_{(\partial \Sigma)_i} x_i, \quad \text{and} \quad f_i:=\frac{\int_{(\partial \Sigma)_i} x_i}{\left|\int_{(\partial \Sigma)_i} x_i\right|}\in \mathbb{S}^2.
\end{align}

Using Lemma \ref{lemm: perpendicular to constant}, we have 
\begin{align*}
    \mathbf{0}=\int_{\partial \Sigma}x =\int_{\partial \Sigma}(x_1, x_2,x_3).
\end{align*}
Thus we have
\begin{align}\label{eqn: flux}
    F_1+\dots F_n=0.
\end{align}

\begin{lemma}\label{lemm: properties of the boundary}Assume $n\ge 2$.
Each $(\partial \Sigma)_i, i=1,\dots ,n$, satisfies the followings.
\begin{enumerate}
    \item \label{lemm: properties of the boundary: strictly convex} $(\partial \Sigma)_i$ is a strictly convex curve in $\mathbb{S}^2$. In other words, a plane passing through the origin intersects $(\partial \Sigma)_i$ in at most two points.
    \item \label{lemm: properties of the boundary: small neighborhood} There is a small neighborhood of $(\partial \Sigma)_i$ in $\Sigma$ that is contained outside of the convex cone over $(\partial \Sigma)_i$. 
    \item \label{lemm: properties of the boundary: intersection with a plane} A plane parallel to the line passing through $f_i$ and the origin intersects $(\partial \Sigma)_i$ in at most two points.
    \item \label{lemm: properties of the boundary: nodal points} Let $u$ be a first Steklov eigenfunction of $\Sigma$. Then, $(\partial \Sigma)_i$ contains at most two nodal points of $u$.
    
\end{enumerate}
\end{lemma}
\begin{proof}
    For the proof of (\ref{lemm: properties of the boundary: strictly convex}), see \cite[Lemma 4.2]{Seo2021SSC} or \cite[Corollary 4.2]{KM20arXiv}. The major tool of the proof is the two-piece property (Theorem \ref{thm: two-piece property}). For (\ref{lemm: properties of the boundary: small neighborhood}), see \cite[Lemma 6.4]{Seo2021SSC}. (\ref{lemm: properties of the boundary: intersection with a plane}) and (\ref{lemm: properties of the boundary: nodal points}) follow from \cite[Proposition 3.6]{Seo2021SSC} and \cite[Remark 4.4]{Seo2021SSC}, respectively.
\end{proof}

\subsection{Properties under $\sigma_1(M)<1$}

\begin{lemma}\label{lemm: perpendicular to constant}
\begin{enumerate}
    \item Any coordinate functions $x_i, i=1,2,3$, are Steklov eigenfunctions of $M$ with eigenvalue 1. In particular, $\sigma_1(M)\le 1$. \label{lemm: perpendicular to constant: coordinate function}
    \item Let $u$ be a first Steklov eigenfunction of $M$. Then,
    \begin{align*}
        \int_{\partial M} u =0.
    \end{align*}
    If $\sigma_1(M)<1$, we have 
    \begin{align*}
        \int_{\partial M} ux_i=0.
    \end{align*} \label{lemm: perpendicular to constant: first eigenfunction}
\end{enumerate}
\end{lemma}
\begin{proof}
    See, for example, \cite[Lemma 2.1]{Seo2021SSC}.
\end{proof}

Note that Lemma \ref{lemm: perpendicular to constant} holds for all compact immersed free boundary minimal submanfolds in $\mathbb{B}^m$ with all coordinate functions $x_1, \dots ,x_m$.

\begin{lemma}[symmetry of a first eigenfunction, {\cite[Lemma 3.2]{McG18ACC}}]\label{lemm: symmetry of a first eigenfunction}
    Assume $\sigma_1(M)<1$.
    Let $u$ be a first Steklov eigenfunction of $M$. If $R_{\Pi}(M)=M$, $u$ is invariant under $R_{\Pi}$.
\end{lemma}

\begin{lemma}[constant sign]\label{lemm: constant sign} Assume $\sigma_1(\Sigma)<1$. If a boundary component $(\partial \Sigma)_i$ is invariant under the reflections through the two distinct planes, a first Steklov eigenfunction of $\Sigma$ has a constant sign on $(\partial \Sigma)_i$.
\end{lemma}

\begin{proof}
    Let $\Pi_1$ and $\Pi_2$ be two distinct reflection planes for $(\partial \Sigma)_i$.  By symmetry principle \cite[Theorem 1.4]{Seo2021SSC} and Lemma \ref{lemm: prop of a reflection} (\ref{lemm: prop of a reflection: origin}), $\Pi_1$ and $\Pi_2$ are two reflection planes for $\Sigma$ passing through the origin. Since $R_{\Pi_1}(f_i)=R_{\Pi_2}(f_i)=f_i$, they also pass through $f_i$. By a similar argument in Case 1 in the proof of \cite[Theorem 1.2]{Seo2021SSC}, the first eigenfunction has a constant sign on $(\partial \Sigma)_i$.
\end{proof}

\section{Main proof}
   Let $D_i, i=1,\dots ,n$ be the convex sets in $\mathbb{S}^2$ bounded by $(\partial \Sigma)_i$, respectively. Note that each of $D_i, i=1,\dots, n$ is well-defined by Lemma \ref{lemm: properties of the boundary} (\ref{lemm: properties of the boundary: strictly convex}). The following theorem insists that no two nested convex sets exist among $D_i$.
\begin{theorem} \label{thm: not nested}
    There are no two distinct $i,j \in \{1,\dots, n\}$ such that $D_i\subsetneq D_j$ or $D_j\subsetneq D_i$.
\end{theorem}
\begin{proof}
    Assume $D_1 \subset D_2$. Take a coordinate plane $\Pi$ that meets $\partial \Sigma$ transversely and meets $(\partial\Sigma)_1$ and $(\partial \Sigma)_2$. Let $\{(\partial \Sigma)_1,\dots (\partial \Sigma)_m\}$ be the boundary components of $\Sigma$ that has an intersection with $\Pi$, and we denote $(\partial\Sigma)_i\cap \Pi$ by $v_i^1$ and $v_i^2$ for all $i=1,2,\dots ,m$.

     We consider a graph whose vertex set is 
     \begin{align}
         \{v_1^1,\dots, v_m^1\} \cup \{v_1^2,\dots, v_m^2\}
     \end{align}
     and the edges are the arcs in $\Pi\cap \Sigma$ whose endpoints are contained in $\partial \Sigma$. By the transversality property (Lemma \ref{lemm: properties of interior} (\ref{lemm: properties of interior: transversality})), every vertex is contained in exactly one edge. Furthermore, we give a direction on the edge set by following the procedure. 
     \begin{enumerate}
         \item Take the edge containing $v_1^2$ and give a direction from $v_1^2$. Let $v$ be the other vertex in the edge.
         \item Take the vertex $w$ which is contained in the same boundary component with $v$ but not identical to $v$.
         \item Repeat the first and second steps until the vertex $v_1^1$ comes out.
     \end{enumerate}
    By the two-piece property (Theorem \ref{thm: two-piece property}) with the fact that $\Sigma$ has genus zero, every edge has a direction. 

    We claim that the directed graph from $v_1^2$ escapes the convex cone over $D_2\cap \Pi$ by an edge. Note the fact that a small part of an arc of $\Pi\cap \Sigma$ from each of vertices $v_i^1$ and $v_i^2$ is contained outside of the convex cone over $D_i\cap \Pi$ by Lemma \ref{lemm: properties of the boundary} (\ref{lemm: properties of the boundary: small neighborhood}). Since we can give directions on every edge, the directed graph from $v_1^2$ escapes the convex cone over $D_2\cap \Pi$. On the other hand, the embeddedness of $\Sigma$ implies that for every $i$, we have either $v_i^1, v_i^2\in \text{int}(D_2)$ or $v_i^1, v_i^2\notin \text{int}(D_2)$. Therefore, the directed graph escapes by an edge that meets one of the line segments $ov_2^1$ and $ov_2^2$ transversely by Lemma \ref{lemm: properties of interior} (\ref{lemm: properties of interior: transversality}), which completes that proof.

    By the previous claim and the fact in the proof, there is a ray from $o$ that meets $\Sigma$ at least twice. It contracts the radial graphicality of $\Sigma$ (Lemma \ref{lemm: properties of interior} (\ref{lemm: properties of interior: radial graphicality})).

\end{proof}

\begin{corollary} \label{cor: flux components}
   $\{f_1,\dots, f_n\}$ has $n$ distinct elements.
\end{corollary}
\begin{proof}
     Suppose $f_i=f_j$ for $i\neq j$.
     If $D_i\cap D_j=\emptyset$, we can find a coordinate plane that separates $D_i$ and $D_j$. Then, $f_i\neq f_j$, a contradiction. On the other hand, we consider $D_i\cap D_j\neq\emptyset$. Since $\Sigma$ is embedded, we have either $D_i\subsetneq D_j$ or $D_j\subsetneq D_i$, which contradicts Theorem \ref{thm: not nested}.
\end{proof}

\begin{corollary} \label{cor: f_i and a boundary component}
    Let $\Pi$ be a reflection plane of $\Sigma$. Then, for $1\le i\le j\le n$, $R_{\Pi}((\partial \Sigma)_i)=(\partial \Sigma)_j$ if and only if $R_{\Pi}(f_i)=f_j$.
\end{corollary}
\begin{proof}
    It is clear that $R_{\Pi}((\partial \Sigma)_i)=(\partial \Sigma)_j$ implies $R_{\Pi}(f_i)=f_j$. Now suppose $R_{\Pi}(f_i)=f_j$. We note that $R_{\Pi}((\partial \Sigma)_i)$ is a boundary component of $\Sigma$, say $(\partial \Sigma)_k$. Then, $f_k = R_{\Pi}(f_i)$, and by the assumption, $f_k=f_j$. By Corollary \ref{cor: flux components}, we have $k=j$ and $R_{\Pi}((\partial \Sigma)_i)=(\partial \Sigma)_j$. 
\end{proof}

\begin{lemma}\label{lemm: transitivity}  Let $\Pi_1,\dots, \Pi_n$ be distinct reflection planes of $\Sigma$. Then, the reflection group generated by $R_{\Pi_1},\dots, R_{\Pi_n}$ acts transitively on $\{f_1, \dots , f_n\}$. Equivalently, the reflection group acts transitively on $\{(\partial \Sigma)_1, \dots , (\partial \Sigma)_n\}$.
\end{lemma}

\begin{proof}
    Suppose the first statement is not true. We consider a graph with vertices $\{f_1,\dots, f_n\}$ and edges $(f_if_j), 1\le i,j\le n$ if there is a plane $\Pi\in\{\Pi_1,\dots, \Pi_n\}$ such that $R_{\Pi}(f_i)=f_j$. By the assumption, this graph has at least two components. Take one of the smallest components with $k$ vertices. Then, $k\le n-k$.

We show that $k\ge 2$. Suppose $k=1$ and let $f_1$ be such vertex. Then, every reflection, $R_{\Pi_1},\dots, R_{\Pi_n}$, maps $f_1$ into $f_1$. This $n$ reflections map $f_2$ into $\{f_1,\dots, f_n\}\setminus\{f_1\}$. If $f_2\neq -f_1$, by Corollary \ref{cor: flux components}, only one coordinate plane passes through $f_1$ and $f_2$. It implies that there is at most one reflection in $\{R_{\Pi_1},\dots, R_{\Pi_n}\}$ that map $f_2$ into $f_2$. In addition, using Lemma \ref{lemm: prop of a reflection} (\ref{lemm: prop of a reflection: identical}), there are at most $n-2$ reflections that map $f_2$ into $\{f_1,\dots, f_n\}\setminus\{f_1, f_2\}$. Thus, there are at most $n-1$ distinct reflections in $\{R_{\Pi_1},\dots, R_{\Pi_n}\}$, a contradiction. Thus, $f_2=-f_1$. In a similar argument, we can show that $f_2=f_3=\cdots=f_n$, which is $-f_1$. It contradicts Corollary $\ref{cor: flux components}$. Therefore, $k\ge 2$.

Let $\{f_1,\dots, f_k\}$ be the vertex set of the smallest component. By Lemma \ref{lemm: prop of a reflection} (\ref{lemm: prop of a reflection: identical}), there are at most $k-1$ reflection planes in $\{\Pi_1,\dots, \Pi_n\}$ that do not pass through $f_1$. Therefore, there must be at least $n-k+1$ reflection planes in $\{\Pi_1,\dots, \Pi_n\}$ that pass through $f_1$. Suppose $k=2$ and $f_2=-f_1$. In this case, there is exactly one reflection plane in $\{\Pi_1,\dots, \Pi_n\}$ that does not pass through $f_1$. Because every reflection plane passing through $f_1$ maps $f_3$ into $\{f_3,\dots, f_n\}$, Lemma \ref{lemm: prop of a reflection} further implies that there can be at most $n-2$ distinct reflection planes passing through $f_1$ in $\{\Pi_1,\dots, \Pi_n\}$. Then, there are at most $n-1$ reflections in $\{\Pi_1,\dots, \Pi_n\}$, contrary to Corollary \ref{cor: flux components}. Therefore we can take $f_2\neq -f_1$, and we can assume $f_2$ is connected to $f_1$ by an edge. By Lemma \ref{lemm: prop of a reflection} (\ref{lemm: prop of a reflection: identical}),  the reflection planes in $\{\Pi_1,\dots, \Pi_n\}$ that pass through $f_1$ maps $f_2$ one-to-one into $\{f_2,\dots, f_k\}$. Thus, there are at most $k-1$ such reflections. However, $n-k+1>k-1$ leads to a contradiction.

The second statement follows from the first statement with an argument in the proof of Corollary \ref{cor: f_i and a boundary component}. 
\end{proof}

Now we prove the main theorem.

\begin{proof}[Proof of Theorem \ref{thm: n symmetries}]
The case for $n=1$ is trivial by the classical work by Nitsche \cite{Nit85SPC}. For $n=2$, our theorem is known in \cite[Theorem 1.2]{Seo2021SSC}. Thus, we only need to consider $n>2$. We prove this theorem by proof by contradiction. We assume $\sigma_1(\Sigma)<1$. We will obtain a contradiction by following the steps.\\

\textit{Step 1.} Let $\Pi_1,\dots, \Pi_n$ be the given reflection planes of $\Sigma$. We show that for each $1\le i,j\le n$, there is $\Pi \in \{\Pi_1, \dots, \Pi_n\}$ such that $R_{\Pi}((\partial \Sigma)_i)=(\partial \Sigma)_j$. By Corollary \ref{cor: f_i and a boundary component} and Lemma \ref{lemm: prop of a reflection} (\ref{lemm: prop of a reflection: identical}), it is sufficient to prove that there is exactly one reflection plane in $\{\Pi_1, \dots, \Pi_n\}$ that passes through each $f_i$. Suppose there are no reflection planes in $\{\Pi_1, \dots, \Pi_n\}$ that pass through $f_i$. Then, by Lemma \ref{lemm: prop of a reflection} (\ref{lemm: prop of a reflection: identical}), there are at most $n-1$ distinct reflections in $\{R_{\Pi_1},\dots, R_{\Pi_n}\}$, a contradiction. If there are two reflections in $\{R_{\Pi_1},\dots, R_{\Pi_n}\}$ that pass through $f_i$, by Lemma \ref{lemm: constant sign}, a first eigenfunction $u$ has a constant sign on $(\partial \Sigma)_i$. By Lemma \ref{lemm: transitivity} with Corollary \ref{cor: f_i and a boundary component} and Lemma \ref{lemm: symmetry of a first eigenfunction}, $u$ has a constant sign on $\partial \Sigma$. It contradicts Lemma \ref{lemm: perpendicular to constant} (\ref{lemm: perpendicular to constant: first eigenfunction}). \\

By Step 1, for each $1\le i, j \le n$, we can define $\Pi_{ij}\in\{\Pi_1,\dots, \Pi_n\}$ and $R_{ij}:= R_{\Pi_{ij}}$ such that 
\begin{align}
 R_{ij}((\partial \Sigma)_i)=(\partial \Sigma)_j.   
\end{align}
In addition, we show that each of the boundary components of $\Sigma$ has exactly two nodal points of $u$. By Lemma \ref{lemm: symmetry of a first eigenfunction} and Lemma \ref{lemm: perpendicular to constant} (\ref{lemm: perpendicular to constant: first eigenfunction}), $u$ changes its sign on each of the boundary components. In addition, every $(\partial\Sigma)_i$ has the same number of nodal points of $u$, and the number is at least two. By Lemma \ref{lemm: properties of the boundary} (\ref{lemm: properties of the boundary: nodal points}), the number should be two. 

Let $p_i^1, p_i^2$ be the nodal point of $u$ in $(\partial \Sigma)_i$. Then, by Lemma \ref{lemm: symmetry of a first eigenfunction}, $R_{ii}(p_i^1)=p_i^2$. Otherwise, $R_{ii}(p_i^1)=p_i^1$ and $R_{ii}(p_i^2)=p_i^2$. Then, by Lemma \ref{lemm: symmetry of a first eigenfunction} again, $u$ cannot change its sign on $(\partial \Sigma)_i$, a contradiction. These nodal points divide each of the boundary components into two arcs. \\

\textit{Step 2.} We show that the two arcs of $(\partial \Sigma)_i\setminus \{p_i^1, p_i^2\}$ are not congruent. Suppose not. Then, either a nontrivial rotation or a reflection that maps one component of $(\partial \Sigma)_i\setminus\{p_i^1, p_i^2\}$ into the other and vice versa. If there is such a reflection, it cannot be $R_{ii}$. By Lemma \ref{lemm: constant sign}, $u$ has a constant sign on $(\partial \Sigma)_i$, a contradiction. If there is such a rotation $A$, $A(f_i)=f_i$ and $A(\{p_i^1, p_i^2\})=\{p_i^1, p_i^2\}$. For the case of $A(p_i^1)=p_i^1$ and $A(p_i^2)=p_i^2$, the rotation axis for $A$ passes through $f_i, p_i^1,$ and $p_i^2$. Since the axis meets $\mathbb{S}^2$ at exactly two points, $f_i$ should be either $p_i^1$ or $p_i^2$. However, $f_i$ lies in the interior of $D_i$ \cite[Proposition 3.5]{Seo2021SSC}, a contradiction. For the case of $A(p_i^1)=p_i^2$ and $A(p_i^2)=p_i^1$, we note that $A$ is the 180\textdegree rotation through the axis passing through the origin and $f_i$. Together with $R_{ii}$, $(\partial \Sigma)_i$ has a reflection symmetry distinct from $R_{ii}$. Lemma \ref{lemm: constant sign} again gives a contradiction.\\

We prove the following step using an argument on arcs.\\

\textit{Step 3.} We show that $\Pi_1, \dots, \Pi_n$ share a common axis. By the previous step, let $\widehat{p_1^1 p_1^2}$ and $\widehat{p_1^1 p_1^2}'$ be the components of $(\partial \Sigma)_1\setminus \{p_1^1, p_i^2\}$. Furthermore, we define an arc $\widehat{p_i^1 p_i^2}$ of $(\partial \Sigma)_i$ that is congruent to $\widehat{p_1^1 p_1^2}$. Clearly,
\begin{align}
    R_{ij}\left(\widehat{p_i^1 p_i^2}\right)= \widehat{p_j^1 p_j^2}.
\end{align} Then, we have
\begin{align}
    R_{2i}\circ R_{12}\left(\widehat{p_1^1 p_1^2}\right) = R_{1i}\circ R_{11} \left(\widehat{p_1^1 p_1^2}\right)=\widehat{p_i^1 p_i^2}.
\end{align}
For simplicity, let $A:=(R_{2i}\circ R_{12})^{-1}\circ (R_{1i}\circ R_{11})$ and let $q_1:= \widehat{p_1^1 p_1^2} \cap \Pi_{11}$. Then, $A(q_1)=q_1$ and $A(p_1^1)$ is either $p_1^1$ or $p_1^2$. Note that $R_{2i}\circ R_{12}$, $ R_{1i}\circ R_{11}$, and $A$ are all rotations along axes passing through the origin.  Then, $A$ is a rotation along the axis passing through the origin and $q_1$. Suppose $A$ is not the identity map. Since the axis and $\widehat{p_1^1 p_1^2}$ are invariant under $A$ and $R_{11}$, each of components of $\widehat{p_1^1 p_1^2}\setminus\{q_1\}$ lies in a plane that passes through the axis. Similarly, $\widehat{p_1^1 p_1^2}'$ lies in two planes passing through the axis. Since $(\partial \Sigma)_1$ is real analytic (see \cite[Section 6]{Lew1951MSP}), this boundary component lies in a plane,  so it is a circle. By \cite[Remark 6.2]{Seo2021SSC}, $\Sigma$ is congruent to one of an equatorial disk and a critical catenoid. It contradicts $n>2$. Thus $A$ is the identity map. Then, the rotation axes of $R_{2i}\circ R_{12}$ and $R_{1i}\circ R_{11}$ are identical. Since each of these axes is the intersection of the two reflection planes, $\Pi_{11}, \Pi_{12},$ and $\Pi_{1i}$ contain a common axis. Since $i$ is arbitrary and
\begin{align}
    \left\{\Pi_{11},\dots, \Pi_{1n}\right\}=\left\{\Pi_1,\dots, \Pi_n\right\},
\end{align}
we obtained our desired conclusion. \\

\textit{Step 4.} Let $x_3$-axis be the common axis of $\Pi_1,\dots, \Pi_n$ as in the previous step. By Step 1 and Lemma \ref{lemm: symmetry of a first eigenfunction}, every nodal point of $u$ in $\partial \Sigma$ has the same $x_3$-coordinates, say $c$. Moreover, using (\ref{eqn: flux}), we have
\begin{align}
    f_1, f_2, \dots, f_n \in \{x_3=0\}
\end{align}
By Lemma \ref{lemm: properties of the boundary} (\ref{lemm: properties of the boundary: intersection with a plane}), $u$ does not change its sign on each of $\partial \Sigma \cap \{x_3>c\}$ and $\partial \Sigma \cap \{x_3<c\}$. Thus, we have
\begin{align}
    \int_{\partial \Sigma}u(x_3-c)\neq 0,
\end{align}
which contradicts Lemma \ref{lemm: perpendicular to constant} (\ref{lemm: perpendicular to constant: first eigenfunction}).\\

Throughout the steps, we obtained a contradiction by assuming $\sigma_1(\Sigma)<1$. By Lemma \ref{lemm: perpendicular to constant} (\ref{lemm: perpendicular to constant: coordinate function}), we have $\sigma_1(\Sigma)=1$.

\end{proof}

For a surface of genus zero, the first eigenvalue of this surface has multiplicity at most 3 \cite[Theorem 2.3]{FS16SEB}. Main theorem implies the following.

\begin{corollary}
    If $\Sigma$ has $n$ reflection symmetries, the first Steklov eigenspace for $\Sigma$ is spanned by the three coordinate functions of $\mathbb{R}^3$.
\end{corollary}

\bibliographystyle{amsplain}
\bibliography{annot}

\end{document}